\documentclass[11pt,reqno]{amsart}
\usepackage{graphicx}
\usepackage{float}
\usepackage{amsfonts}
\usepackage[caption = false]{subfig}
\usepackage{amsmath}
\usepackage{graphics}
\usepackage{float}
\usepackage{enumerate}
\usepackage{mathtools}
\usepackage{amsthm}
\usepackage[english]{babel}
\usepackage{amsfonts,amssymb}
\usepackage{mathrsfs}
\usepackage[utf8x]{inputenc}\usepackage[english]{babel}
\usepackage{soul}
\usepackage[all]{xy,xypic}
\usepackage{cite}
\usepackage{relsize}
\numberwithin{equation}{section}
\newtheorem{theorem}{Theorem}[section]
\newtheorem{corollary}[theorem]{Corollary}

\theoremstyle{definition}
\newtheorem{definition}[theorem]{Definition}
\theoremstyle{remark}

\numberwithin{equation}{section}
\DeclareMathOperator{\RE}{Re}

\newtheorem{theo}{Lemma}
\newcounter{tmp}

\begin{document}
	\title[ Differential Subordination of certain class of starlike functions ]{Differential Subordination of certain class of starlike functions}
	
	\author[Neha Verma]{Neha Verma}
	\address{Department of Applied Mathematics, Delhi Technological University, Delhi--110042, India}
	\email{nehaverma1480@gmail.com}

	\author[S. Sivaprasad Kumar]{S. Sivaprasad Kumar}
	\address{Department of Applied Mathematics, Delhi Technological University, Delhi--110042, India}
	\email{spkumar@dce.ac.in}

	\subjclass[2010]{30C45, 30C80}
	
	\keywords{Differential Subordination, Exponential, Starlike}
	\maketitle
\begin{abstract}
This paper presents several results concerning second and third-order differential subordination for the class $\mathcal{S}^{*}_{e}:=\{f\in \mathcal{A}:zf'(z)/f(z)\prec e^z\}$, which represents the class of starlike functions associated with exponential function.

\end{abstract}
\maketitle
	
\section{Introduction}
	\label{intro}
\noindent Suppose $\mathcal{H}=\mathcal{H}(\mathbb{D})$ be the class of analytic functions in $\mathbb{D}:=\{z\in \mathbb{C}:|z|<1\}$. For any positive integer $n$ and $a\in \mathbb{C}$, define
    $$\mathcal{H}[a,n]:=\{f\in \mathcal{H}:f(z)=a+a_nz^n+a_{n+1}z^{n+1}+a_{n+2}z^{n+2}+\cdots\}.$$
Let $\mathcal{A}$ be the class of normalized analytic functions defined as 
$$\mathcal{A}=\{f\in \mathcal{H}:f(z)=z+a_2z^2+a_3z^3+\cdots\}$$
and $\mathcal{S}\subset\mathcal{A}$ contains univalent functions. Any function $f\in \mathcal{S}$ is called starlike if and only if $\RE(zf'(z)/f(z))>0$ and the class consisting of all such functions is denoted by $\mathcal{S}^{*}$. Let $h$ and $g$ are two functions in the class $\mathcal{A}$, we say that $h$ is subordinate to $g$ (denoted by $h\prec g$), if there exists a Schwarz function $w$ with $w(0)=0$ and $|w(z)|\leq |z|$ such that $h(z)=g(w(z))$. Moreover, if $g$ is univalent, then $h\prec g$ if and only if $h(0)=g(0)$ and $h(\mathbb{D})\subseteq g(\mathbb{D})$. A significant development in the direction of differential subordination was made with the appearance of Miller and Mocanu's monograph, ``Differential Subordination and Univalent Functions," in 1981. It is a compendium of its time and this book initiated a significant revolution among researchers. In 1992, Ma and Minda \cite {ma-minda} unified various subclasses of $\mathcal{S}^{*}$ by introducing the following class $\mathcal{S}^{*}(\varphi)$ through subordination:
\begin{equation}
		\mathcal{S}^{*}(\varphi)=\bigg\{f\in \mathcal {A}:\dfrac{zf'(z)}{f(z)}\prec \varphi(z) \bigg\},\label{mindaclass}
\end{equation}
with $\varphi$ as an analytic univalent function so that $\RE\varphi(z)>0$, $\varphi(\mathbb{D})$ is symmetric about the real axis and starlike with respect to $\varphi(0)=1$ with $\varphi'(0)>0$. Through the process of specializing the function $\varphi$ in \eqref{mindaclass}, researchers have obtained several significant subclasses of $\mathcal{S}^*$. Some of these subclasses are mentioned below for ready reference in Table \ref{8 table1}.

\begin{table}[hbt!]
\caption{Some subclasses of $\mathcal{S}^{*}$}
\centering
\begin{tabular}{|l |l| l|} 
 \hline
 Class& Author(s) & Reference \\ [1ex] 
 \hline
 $\mathcal{S}^*[C,D]:=\mathcal{S}^*((1+Cz)/(1+Dz))$   & Janowski&\cite{1janowski}     \\
 $\mathcal{S}^{*}_{\rho}:=\mathcal{S}^{*}(1+\sinh^{-1}(z))$ & Arora and Kumar&\cite{kush} \\ 
  $\mathcal{S}^{*}_{S}:=\mathcal{S}^{*}(1+\sin z)$ &  Cho et al.&\cite{chosine} \\
  $\mathcal{S}^{*}_{SG}:=\mathcal{S}^{*}(2/(1+e^{-z}))$ & Goel and Kumar &\cite{goel} \\
  $\mathcal{S}^{*}_{\varrho}:=\mathcal{S}^{*}(1+ze^z)$ & Kumar and Kamaljeet &\cite{kumar-ganganiaCardioid-2021}\\
   $\mathcal{S}^{*}_{q}:=\mathcal{S}^{*}(z+\sqrt{1+z^2})$ &  Raina and  Sok\'{o}\l&\cite{raina} \\
 $\mathcal{S}^{*}_ L:=\mathcal{S}^{*}(\sqrt{1+z})$&  Sok\'{o}\l \ and Stankiewicz  &\cite{stan}   \\[1ex] 
 \hline
\end{tabular}
\label{8 table1}
\end{table}

\noindent The exponential function plays a crucial role in diverse scientific fields, including physics, signal processing, number theory, and models describing growth and decay in biology and economics. Specifically, in the realm of GFT, Mendiratta et al. \cite{mendi} introduced and examined the class $\mathcal{S}^{*}_{e}:=\mathcal{S}^{*}(e^z)$, incorporating the use of $e^z$ within their class. Geometrically, we say a function $f\in \mathcal{S}^{*}_{e}$ if and only if $zf'(z)/f(z)\in \Delta_e:=\{\delta\in \mathbb{C}:|\log \delta|<1\}$. In this investigation, they obtained the conditions on $\beta$ parameter so that $p(z)\prec e^z$ when $1+\beta zp'(z)/p(z)\prec h(z)$ by choosing $h(z)=(1+Cz)/(1+Dz)$, $\sqrt{1+z}$ and $e^z$. Further, Kumar and Ravichandran \cite{sushil} extended their results and obtained sharp bounds on $\beta$ so that $p(z)\prec e^z$ when $1+\beta z p'(z)/p^j(z)\prec h(z)$ for $j=0,1,2$ and $h(z)=(1+Cz)/(1+Dz)$ and $\sqrt{1+z}$. Later on, Naz et al. \cite{adibastarlikenessexponential} derived a wide range of first-order subordination implications for the exponential function and generalized the results of \cite{mendi}. Madaan et al. \cite{madaan} conducted an investigation for the class of admissible functions associated with the lemniscate of Bernoulli for several implications of first and second order differential subordination. Similarly, Goel and Kumar \cite{goel} derived results focusing on first-order subordination results considering the class $\mathcal{S}^{*}_{SG}$. To explore the latest developments in this theory, one can refer to \cite{goelhigher,mushtaq}. There do exist good literature on first-order differential subordination, see \cite{adibastarlikenessexponential,ckms} and the references therein. However, the work on second and higher-order differential subordination is still in the process of development. 


Motivated by the aforementioned works and its applications, we establish results on second and third-order differential subordination implications for starlike functions associated with the exponential function. These results are derived through the technique of admissibility conditions.
Our results extend the study of the class $\mathcal{S}^{*}_e$ specifically involving differential subordination implications result. 

Now, we state some basic definitions and notations required for differential subordination.


\begin{definition}\cite{antoninoandmiller}
Let $\xi(r,s,t,u;z):\mathbb{C}^4\times \mathbb{D}\rightarrow\mathbb{C}$ and $h(z)$ be a univalent function in $\mathbb{D}$, if $p$ is an analytic function in $\mathbb{D}$ satisfying the third-order differential subordination
\begin{equation}\label{8 def3}
    \xi(p(z),zp'(z),z^2p''(z),z^3p'''(z);z)\prec h(z)
\end{equation}
then $p$ is called the solution of the differential subordination. The univalent function $q$ is said to be a dominant of the solutions of the differential subordination if $p\prec q$ for all $p$ satisfying \eqref{8 def3}. A dominant $\bar{q}$ that satisfies $\bar{q}\prec q$ for all dominants $q$ of \eqref{8 def3} is said to be the best dominant of \eqref{8 def3}, which is unique upto the rotation.
\end{definition}
Furthermore, let $Q$ denote the set of analytic and univalent functions $q\in\overline{\mathbb{D}}\setminus \mathbb{E}(q)$, where
  $$ \mathbb{E}(q)=\{\zeta\in \partial \mathbb{D}:\lim_{z\rightarrow\zeta}q(z)=\infty\}$$
such that $q'(\zeta)\neq 0$ for $\zeta\in \partial \mathbb{D}\setminus \mathbb{E}(q)$. The subclass of $Q$ for which $q(0)=a$ is denoted by $Q(a)$.

\begingroup
\setcounter{tmp}{\value{theo}}
\setcounter{theo}{0} 
\renewcommand\thetheo{\Alph{theo}}
\begin{theo}\cite{antoninoandmiller}\label{lemmaformk}
Let $z_0\in \mathbb{D}$ and $r_0=|z_0|$. Let $f(z)=a_nz^n+a_{n+1}z^{n+1}+\cdots$ be continuous on $\overline{\mathbb{D}}_{r_0}$ and analytic on $\mathbb{D}\cup\{z_0\}$ with $f(z)\neq 0$ and $n\geq 2$. If $|f(z_0)|=\max \{|f(z)|:z\in \overline{\mathbb{D}}_{r_0}\}$ and $|f'(z_0)|=\max\{|f'(z)|:z\in \overline{\mathbb{D}}_{r_0}\}$, then there exist real constants $m$, $k$ and $l$ such that
\begin{equation*}
      \frac{z_0f'(z_0)}{f(z_0)}=m,\quad 1+\frac{z_0f''(z_0)}{f'(z_0)}=k\quad \text{and}\quad 2+\RE\bigg(\frac{z_0f'''(z_0)}{f''(z_0)}\bigg)=l\quad \text{where}\quad l\geq k\geq m\geq n\geq2.
\end{equation*}
\end{theo}
\endgroup


\section{Preliminaries}

\noindent In 2020, Kumar and Goel \cite{goelhigher} made modifications to the results established by Antonino and Miller \cite{antoninoandmiller}, ensuring that the Ma-Minda functions meet the requirements for third-order differential subordination. The revised results are presented below in the form of a definition and theorem, which are needed to achieve our main results for the class $\mathcal{S}^*_e$. 

\begin{definition}\cite{goelhigher}
Let $\Omega$ be a set in $\mathbb{C}$, $q\in Q$ and $k\geq m\geq n\geq 2$. The class of admissible operators $\xi_n[\Omega,q]$ consists of those $\xi:\mathbb{C}^4\times \mathbb{D}\rightarrow \mathbb{C}$ that satisfy the admissibility condition
\begin{equation*}
 \xi(r,s,t,u;z)\notin \Omega\quad \text{whenever}\quad z\in \mathbb{D},\quad r=q(\zeta),\quad s=m\zeta q'(\zeta),      
\end{equation*}
\begin{equation*}
 \RE\bigg(1+\frac{t}{s}\bigg)\geq m\bigg(1+\RE \frac{\zeta q''(\zeta)}{q'(\zeta)}\bigg)\quad \text{and}\quad \RE \frac{u}{s}\geq m^2\RE \frac{\zeta^2 q'''(\zeta)}{q'(\zeta}+3m(k-1)\RE \frac{\zeta q''(\zeta)}{q'(\zeta)}
 \end{equation*}
for $\zeta\in \partial \mathbb{D}\setminus \mathbb{E}(q)$. 
\end{definition}

\begingroup
\setcounter{tmp}{\value{theo}}
\setcounter{theo}{1} 
\renewcommand\thetheo{\Alph{theo}}
\begin{theo}\cite{goelhigher}\label{8 firsttheoremthirdorder}
  Let $p\in \mathcal{H}[a,n]$ with $m\geq n\geq2$, and let $q\in Q(a)$ such that it satisfies
\begin{equation}
       \bigg|\frac{zp'(z)}{q'(\zeta)}\bigg|\leq m,
\end{equation}
for $z\in \mathbb{D}$ and $\zeta\in \partial \mathbb{D}\setminus \mathbb{E}(q)$. If $\Omega$ is a set in $\mathbb{C}$, $\xi\in \Psi_n[\Omega,a]$ and
$$\xi(p(z),zp'(z),z^2p''(z),z^3p'''(z);z)\subset \Omega,$$
then $p\prec q$.
\end{theo}
\endgroup


Here, we provide two lemmas which are necessary for the proof of results in the coming sections.

\begingroup
\setcounter{tmp}{\value{theo}}
\setcounter{theo}{2} 
\renewcommand\thetheo{\Alph{theo}}
\begin{theo}\cite[Lemma 4, Pg No. 192]{goelhigher}\label{prilemma4}
For any complex number $z$, we have
$$  |\log (1+z)|\geq 1\quad \text{if and only if} \quad |z|\geq e-1.$$
\end{theo}
\endgroup


\begingroup
\setcounter{tmp}{\value{theo}}
\setcounter{theo}{3} 
\renewcommand\thetheo{\Alph{theo}}
\begin{theo}\cite{goel}\label{prilemma42}
Let $r_0\approx 0.546302$ be the positive root of the equation $r^2+2 \cot(1)r-1=0$. Then
$$\bigg|\log \bigg(\frac{1+z}{1-z}\bigg)\bigg|\geq 1\quad \text{on}\quad |z|=R\quad \text{if and only if}\quad R\geq r_0.$$
\end{theo}
\endgroup


\section{Main Results}
\noindent We begin with considering the function $q(z):=e^z$ and define the admissibility class $\Psi[\Omega,q]$, where $\Omega\subset \mathbb{C}$. We know that $q$ is analytic and univalent on $\overline{\mathbb{D}}$ with $q(0)=1$ and it maps $\mathbb{D}$ onto the domain $\Delta_{e}$. Since $\mathbb{E}(q)$ is empty for $\zeta\in \partial \mathbb{D}\setminus \mathbb{E}(q)$ if and only if $\zeta=e^{i\theta}$ $(\theta\in[0,2\pi])$. Now, consider
\begin{equation}
    |q'(\zeta)|=e^{\cos \theta}=:b(\theta)\label{8 1}
\end{equation}
has a minimum value of $e^{-1}$. It is evident that $\min |q'(\zeta)|>0$, which implies that $q\in Q(1)$, and consequently, the admissibility class $\Psi[\Omega,q]$ is well-defined. While considering $|\zeta|=1$, we observe that $q(\zeta)\in q(\partial \mathbb{D})=\partial \Delta_e = \{\delta \in \mathbb{C}: |\log \delta| = 1\}$. Consequently, we have $|\log q(\zeta)| = 1$ and $\log (q(\zeta)) = e^{i\theta}$ ($ \theta\in[0,2\pi]$), implies that $q(\zeta) = e^{\zeta}$. Moreover, $\zeta q'(\zeta)=e^{i\theta} e^{e^{i\theta}}$ and
\begin{equation}
    \frac{\zeta q''(\zeta)}{q'(\zeta)}=e^{i\theta}.\label{8 2}
\end{equation}
By comparing the real parts on both sides of \eqref{8 2}, we obtain
\begin{equation}
    \RE\bigg(\frac{\zeta q''(\zeta)}{q'(\zeta)}\bigg)=\cos \theta=:l(\theta).\label{8 3} 
\end{equation}
The function $l(\theta)$ defined in \eqref{8 3} achieves its minimum value at $\theta = \pi$, given by $l(\pi) \approx -1$. Moreover, the class $\Psi[\Omega, e^z]$ is precisely defined as the class of all functions $\xi: \mathbb{C}^3 \times \mathbb{D} \rightarrow \mathbb{C}$ that satisfy the following conditions:
   $$ \xi(r,s,t;z)\notin \Omega, \quad \text{whenever}$$
\begin{equation}
    r=q(\zeta)=e^{e^{i\theta}};\quad s=m\zeta q'(\zeta)=me^{i\theta}e^{e^{i\theta}};\quad \RE\bigg(1+\frac{t}{s}\bigg)\geq m(1+l(\theta)) \label{8 4}
    \end{equation}
    where $z\in \mathbb{D}, \theta\in[0,2\pi]$ and $m\geq 1$. 
If $\xi:\mathbb{C}^2\times \mathbb{D}\rightarrow \mathbb{C}$, then the admissibility condition \eqref{8 4} becomes 
$$\xi(e^{e^{i\theta}},me^{i\theta}e^{e^{i\theta}};z)\notin \Omega\quad (z\in \mathbb{D}, \theta\in[0,2\pi], m\geq 1).$$

Taking $q(z)=e^z$, we deduce the following as a special case of Miller and Mocanu  \cite[Theorem 2.3b]{miller}, which is necessary to prove our main results.

\begingroup
\setcounter{tmp}{\value{theo}}
\setcounter{theo}{4} 
\renewcommand\thetheo{\Alph{theo}}
\begin{theo}\label{8 theorem1}
Let $\xi\in \Psi[\Omega, e^z]$. If $p\in \mathcal{H}[1,n]$ satisfies
    $$ \xi(p(z),zp'(z),z^2p''(z);z)\in \Omega\quad \text{for} \quad z\in \mathbb{D}$$
then $p(z)\prec e^z$.
\end{theo}
\endgroup

Further, let $q(z)=e^z$, we have
  $$  \zeta^2\frac{q'''(z)}{q'(\zeta)}=e^{2i\theta}.$$
On comparing the real parts of both the sides, we get
   $$ \RE\bigg(\zeta^2\frac{q'''(z)}{q'(\zeta)}\bigg)=\cos 2\theta=:h(\theta).$$ 
The minimum value of $h(\theta)$ is $-1$, which is attained at $\theta=\pi/2$. Thus, if $\xi:\mathbb{C}^4\times \mathbb{D}\rightarrow\mathbb{C}$ then $\xi\in \Psi[\Omega,e^z]$, provided $\xi$ satisfies the following conditions:
\begin{equation*}
    \xi(r,s,t,u;z)\notin \Omega\quad \text{whenever}\quad r=q(\zeta)=e^{e^{i\theta}},\quad s=m\zeta q'(\zeta)=me^{i\theta}e^{e^{i\theta}},
\end{equation*}
\begin{equation*}
   \RE\bigg(1+\frac{t}{s}\bigg)\geq m(1+l(\theta)) \quad\text{and}\quad \RE \frac{u}{s}\geq m^2h(\theta)+3m(k-1)l(\theta) 
\end{equation*}
for $z\in \mathbb{D}$, $\theta\in[0,2\pi]$ and $k\geq m\geq2$. 
Given the above conditions, we deduce a special case of Lemma \ref{8 firsttheoremthirdorder}, which will be implemented in the last section of this article.

\begingroup
\setcounter{tmp}{\value{theo}}
\setcounter{theo}{5} 
\renewcommand\thetheo{\Alph{theo}}
\begin{theo}\label{8 firstlemmathirdorder}
Let $p\in \mathcal{H}[1,n]$ with $m\geq n\geq 2$ such that for $z\in \mathbb{D}$ and $\zeta\in \partial\mathbb{D}$, it satisfies
  $$ |zp'(z)e^{\zeta}|\leq m. $$    
If $\Omega$ is a set in $\mathbb{C}$ and $\xi\in \Psi[\Omega,e^z]$, then
   $$ \xi(p(z),zp'(z),z^2p''(z),z^3p'''(z);z)\subset \Omega\implies p(z)\prec e^z.$$
\end{theo}
\endgroup


Now, we discuss the implications of the following second order differential subordination, given by
\begin{equation}
    1+\alpha_1 zp'(z)+\alpha_2 z^2p''(z)\prec h(z) \implies p(z)\prec e^z.\label{8 secondorderimplication}
\end{equation}
For different choices of $h(z)$, we derive conditions on the constants $\alpha_1$ and $\alpha_2$ to ensure that \eqref{8 secondorderimplication} holds.

Considering $h(z)=(1+Cz)/(1+Dz)$, we have the following result.

\begin{theorem}\label{8 secondorder1}
Suppose $\alpha_1$, $\alpha_2>0$ and $(\alpha_1 -\alpha_2)(1-D^2)\geq e(C-D)(1+|D|)$, where $-1< D<C\leq 1$. Consider the analytic function $p$ in $\mathbb{D}$ such that $p(0)=1$ and
     $$   1+\alpha_1 zp'(z)+\alpha_2 z^2p''(z)\prec \frac{1+Cz}{1+Dz}$$
implies $p(z)\prec e^z$.
\end{theorem}

\begin{proof}
Suppose $h(z)=(1+Cz)/(1+Dz)$ for $z\in \mathbb{D}$ and
     $$   h(\mathbb{D})=\Omega:=\bigg\{\delta\in \mathbb{C}:\bigg|\delta-\frac{1-CD}{1-D^2}\bigg|<\frac{C-D}{1-D^2}\bigg\}.$$
We define $\xi:\mathbb{C}^3\times \mathbb{D}\rightarrow \mathbb{C}$ as $\xi(\eta_1,\eta_2,\eta_3;z)=1+\alpha_1 \eta_2+\alpha_2 \eta_3$. It is clear that $\xi\in \Psi[\Omega,\Delta_e]$ if $\xi(r,s,t;z)\notin \Omega$ for $z\in \mathbb{D}$. Consider
\begin{align*}
        \bigg|\xi(r,s,t;z)-\frac{1-CD}{1-D^2}\bigg|&=\bigg|1+\alpha_1 s+\alpha_2 t-\frac{1-CD}{1-D^2}\bigg|\\
        &\geq |\alpha_1 s+\alpha_2 t|-\frac{|D|(C-D)}{1-D^2}\\
        &\geq |\alpha_1 s|\RE\bigg(1+\frac{\alpha_2}{\alpha_1}\frac{t}{s}\bigg)-\frac{|D|(C-D)}{1-D^2}\\
        &\geq m\alpha_1 b(\theta)\RE\bigg(1+\frac{\alpha_2}{\alpha_1}(ml(\theta)+m-1)\bigg)-\frac{|D|(C-D)}{1-D^2}.
\end{align*}
Here $b(\theta)$ and $l(\theta)$ are given in \eqref{8 1} and \eqref{8 3}. As $m\geq 1$, therefore
\begin{align*}
   \bigg|\xi(r,s,t;z)-\frac{1-CD}{1-D^2}\bigg|&\geq  \alpha_1 b(\theta)\RE\bigg(1+\frac{\alpha_2}{\alpha_1}l(\theta)\bigg)-\frac{|D|(C-D)}{1-D^2}\\
   &\geq\frac{1}{e}\bigg(\alpha_1-\alpha_2\bigg)-\frac{|D|(C-D)}{1-D^2}\\
   &\geq \frac{C-D}{1-D^2}.
\end{align*}
Thus, $\xi\in \Psi[\Omega,\Delta_e]$. Furthermore, $p(z)\prec e^z$ through Lemma \ref{8 theorem1}.
\end{proof}

If we consider $p(z)=zf'(z)/f(z)$ in above theorem, we get the following result:
\begin{corollary}\label{8 a1a2a3}
Suppose $\alpha_1$, $\alpha_2>0$ and $f\in\mathcal{A}$. Let
\begin{align}
Y_f(z)&:=1+\alpha_1(A_2-A_1^2+A_1)+\alpha_2(A_3+2A_2+2A_1^3-2A_1^2-3A_1A_2).  \label{8 corollary21}
\end{align}
where $A_1:=zf'(z)/f(z)$, $A_2:=z^2f''(z)/f(z)$ and $A_3:=z^3f'''(z)/f(z)$.
Then, $f(z)\in \mathcal{S}^{*}_{e}$ if 
 $Y_{f}(z)\prec (1+Cz)/(1+Dz)$ and $(\alpha_1 -\alpha_2)(1-D^2)\geq e(C-D)(1+|D|)$ whenever $-1< D<C\leq 1$.
\end{corollary}

\begin{theorem}
Suppose $\alpha_1$, $\alpha_2>0$ and ${\alpha_1}^2-2\alpha_1 \alpha_2+{\alpha_2}^2-2e\alpha_1+2e\alpha_2\geq e^2$.  
Consider the analytic function $p$ in $\mathbb{D}$ such that $p(0)=1$ and
   $$   1+\alpha_1 zp'(z)+\alpha_2 z^2 p''(z)\prec \sqrt{1+z}$$
implies $p(z)\prec e^z$.
\end{theorem}

\begin{proof}
Suppose $h(z)=\sqrt{1+z}$ for $z\in \mathbb{D}$ and $h(\mathbb{D})=\Omega:=\{\delta\in \mathbb{C}:|\delta^2-1|<1\}$. Let us define $\xi:\mathbb{C}^3\times \mathbb{D}\rightarrow \mathbb{C}$ as $\xi(\eta_1,\eta_2,\eta_3;z)=1+\alpha_1 \eta_2+\alpha_2 \eta_3$. It is clear that $\xi\in \Psi[\Omega,\Delta_e]$ if $\xi(r,s,t,;z)\notin \Omega$ for $z\in \mathbb{D}$. Note that
\begin{align*}
        |(\xi(r,s,t;z))^2-1|&=|(1+\alpha_1 s+\alpha_2 t)^2-1|\\
        &\geq |\alpha_1 s+\alpha_2 t|(|\alpha_1 s+\alpha_2 t|-2)\\
        &\geq |\alpha_1 s|\RE\bigg(1+\frac{\alpha_2}{\alpha_1}\frac{t}{s}\bigg)\bigg(|\alpha_1 s|\RE\bigg(1+\frac{\alpha_2}{\alpha_1}\frac{t}{s}\bigg)-2\bigg)
\end{align*}
Using the fact that $m\geq 1$ and analogous to Theorem \ref{8 secondorder1}, we have
\begin{align*}
       |(\xi(r,s,t;z))^2-1|&\geq \frac{1}{e}\bigg(\alpha_1-\alpha_2\bigg) \bigg(\frac{1}{e}\bigg(\alpha_1-\alpha_2\bigg)-2\bigg)\\
       &=\frac{{\alpha_1}^2-2\alpha_1 \alpha_2+{\alpha_2}^2-2e\alpha_1+2e\alpha_2}{e^2}\\
       &\geq 1.
\end{align*}
Therefore, $\xi\in \Psi[\Omega,\Delta_e]$ and thus $p(z)\prec e^z$ as a consequence of Lemma \ref{8 theorem1}.
\end{proof}

\begin{corollary}
By considering $p(z)=zf'(z)/f(z)$ in above theorem and through \eqref{8 corollary21}, we deduce that $f(z)\in \mathcal{S}^{*}_e$ if
    $Y_{f}(z)\prec \sqrt{1+z}$ and ${\alpha_1}^2-2\alpha_1 \alpha_2+{\alpha_2}^2-2e\alpha_1+2e\alpha_2\geq e^2$.
\end{corollary}

Next, we consider $h(z)=2/(1+e^{-z})$ and $z+\sqrt{1+z^2}$ in the coming results.

\begin{theorem}\label{8 thmsigmoid}
Suppose $\alpha_1$, $\alpha_2>0$ and $\alpha_1-\alpha_2\geq e r_0$, where $r_0\approx 0.546302$ is the positive root of the equation $r^2+2 \cot (1)r-1=0$. Consider the analytic function $p$ in $\mathbb{D}$ such that $p(0)=1$ and
$$1+\alpha_1 zp'(z)+\alpha_2 z^2p''(z)\prec \frac{2}{1+e^{-z}} $$      
implies $p(z)\prec e^z$.
\end{theorem}

\begin{proof}
Suppose $h(z)=2/(1+e^{-z})$ and $h(\mathbb{D})=\Omega:=\{\delta\in \mathbb{C}:|\log (\delta/(2-\delta))|<1\}$. Define $\xi:\mathbb{C}^3\times \mathbb{D}\rightarrow \mathbb{C}$ as $\xi(\eta_1,\eta_2,\eta_3;z)=1+\alpha_1 \eta_2+\alpha_2 \eta_3$. We know that $\xi\in \Psi[\Omega, \Delta_e]$ only when $\xi(r,s,t;z)\notin \Omega$. First we consider,
\begin{align*}
    |\alpha_1 s+\alpha_2 t|&=\alpha_1 |s|\bigg|1+\frac{\alpha_2}{\alpha_1}\frac{t}{s}\bigg|\\
    &\geq \alpha_1 |s|\RE\bigg(1+\frac{\alpha_2}{\alpha_1}\frac{t}{s}\bigg)\\
    &\geq m \alpha_1 b(\theta)\bigg(1+\frac{\alpha_2}{\alpha_1}(ml(\theta)+m-1)\bigg).
\end{align*}
As $m\geq 1$, we obtain
\begin{align}\label{8 7}
    |\alpha_1 s+\alpha_2 t|&\geq b(\theta)(\alpha_1+\alpha_2 l(\theta))\nonumber\\
    &\geq \frac{1}{e}(\alpha_1-\alpha_2)\nonumber\\
    &\geq r_0.
\end{align}
Now, we consider
\begin{equation*}
    \bigg|\log \bigg(\frac{\xi(r,s,t;z)}{2-\xi(r,s,t;z)}\bigg)\bigg|=\bigg|\log\bigg(\frac{1+\alpha_1 s+\alpha_2 t}{1-(\alpha_1 s+\alpha_2 t)}\bigg)\bigg|.
\end{equation*}
Through Lemma \ref{prilemma42} and \eqref{8 7}, we have
\begin{equation*}
    \bigg|\log \bigg(\frac{1+\alpha_1 s+\alpha_2 t}{1-(\alpha_1 s+\alpha_2 t)}\bigg)\bigg|\geq 1,
\end{equation*}
which implies that $\xi\in \Psi[\Omega,\Delta_e]$. Therefore, $p(z)\prec e^z$ using Lemma \ref{8 theorem1}.
\end{proof}

\begin{theorem}
Suppose $\alpha_1$, $\alpha_2>0$ and $\alpha_1-\alpha_2\geq \sqrt{2}e$. Let $p$ be an analytic function in $\mathbb{D}$ such that $p(0)=1$ and
    $$    1+\alpha_1 zp'(z)+\alpha_2 z^2p''(z)\prec z+\sqrt{1+z^2}$$
implies $p(z)\prec e^z$.
\end{theorem}

\begin{proof}
Suppose $h(z)=z+\sqrt{1+z^2}$ for $z\in \mathbb{D}$ and $h(\mathbb{D})=\Omega:=\{\delta\in \mathbb{C}:|\delta^2-1|<2|\delta|\}$. We define $\xi:\mathbb{C}^3\times \mathbb{D}\rightarrow \mathbb{C}$ as $\xi(\eta_1,\eta_2,\eta_3;z)=1+\alpha_1 \eta_2+\alpha_2 \eta_3$. For $\xi\in \Psi[\Omega, \Delta_e]$, we must have $\xi(r,s,t;z)\notin \Omega$. From the graph of $z+\sqrt{1+z^2}$ (see Fig. \ref{crescent}), we note that $\Omega$ is constructed by the circles $C_1$ and $C_2$, given by
     $$   C_1:|z-1|=\sqrt{2}\quad\text{and}\quad C_2:|z+1|=\sqrt{2}.$$
\begin{figure}[h]
\centering
   \includegraphics[width=8cm, height=4cm]{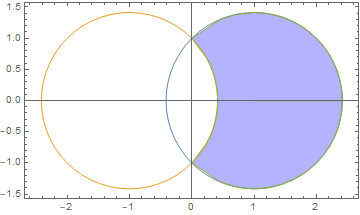}
 \caption{Graph of two circles, namely $C_1$ (blue boundary) and $C_2$ (orange boundary). While the shaded region (solid purple) represents $z+\sqrt{1+z^2}$.}
  \label{crescent}
\end{figure}
It is obvious that $\Omega$ contains the disk enclosed by $C_1$ and excludes the portion of the disk enclosed by $C_2\cap C_1$. We have
      $$  |\xi(r,s,t;z)-1|=|\alpha_1 s+\alpha_2 t|.$$
Analogous to Theorem \ref{8 thmsigmoid}, we have
\begin{align*}
       |\alpha_1 s+\alpha_2 t|&\geq b(\theta)(\alpha_1+\alpha_2 l(\theta))\\
       &\geq \frac{1}{e}(\alpha_1-\alpha_2 )\\
       &\geq \sqrt{2}.
\end{align*}
The fact that $\xi(r,s,t;z)\notin C_1$ suffices us to deduce that $\xi(r,s,t;z)\notin \Omega$. 
Consequently, $\xi\in \Psi[\Omega,\Delta_e]$ and thus $p(z)\prec e^z$ by using Lemma \ref{8 theorem1}.
\end{proof}

Now, if we choose $h(z)$ to be $1+\sin z$ and $1+ze^z$, we derive the conditions on $\alpha_1$ and $\alpha_2$ as follows:

\begin{theorem}
Suppose $\alpha_1$, $\alpha_2>0$ and $\alpha_1-\alpha_2\geq e\sinh 1$. Let $p$ be an analytic function in $\mathbb{D}$ such that $p(0)=1$ and
     $$   1+\alpha_1 zp'(z)+\alpha_2 z^2p''(z)\prec 1+\sin z$$
implies $p(z)\prec e^z$.
\end{theorem}

\begin{proof}
Suppose $h(z)=1+\sin z$ for $z\in \mathbb{D}$ and $h(\mathbb{D})=\Omega:=\{\delta\in \mathbb{C}:|\arcsin (\delta-1)|<1\}$. Define $\xi:\mathbb{C}^3\times \mathbb{D}\rightarrow \mathbb{C}$ as $\xi(\eta_1,\eta_2,\eta_3;z)=1+\alpha_1 \eta_2+\alpha_2 \eta_3$. For $\xi\in \Psi[\Omega, \Delta_e]$, we must have $\xi(r,s,t;z)\notin \Omega$. Through \cite[Lemma 3.3]{chosine}, we note that the smallest disk containing $\Omega$ is $\{w\in \mathbb{C}: |w-1|<\sinh 1\}$. So,
     $$   |\xi(r,s,t;z)-1|=|\alpha_1 s+\alpha_2 t|.$$
Analogous to Theorem \ref{8 thmsigmoid}, we have
\begin{align*}
       |\alpha_1 s+\alpha_2 t|&\geq b(\theta)(\alpha_1+\alpha_2 l(\theta))\\
       &\geq \frac{1}{e}(\alpha_1-\alpha_2)\\
       &\geq \sinh 1.
\end{align*}
Clearly, $\xi(r,s,t;z)\notin\{w: |w-1|<\sinh 1\}$ which is enough to conclude that $\xi(r,s,t;z)\notin \Omega$. Therefore, $\xi\in \Psi[\Omega,\Delta_e]$ and thus $p(z)\prec e^z$ through Lemma \ref{8 theorem1}.
\end{proof}

\begin{theorem}
Suppose $\alpha_1$, $\alpha_2>0$ and $\alpha_1-\alpha_2\geq e^2$. Consider the analytic function $p$ in $\mathbb{D}$ such that $p(0)=1$ and
      $$  1+\alpha_1 zp'(z)+\alpha_2 z^2p''(z)\prec 1+ze^z$$
implies $p(z)\prec e^z$.
\end{theorem}

\begin{proof}
Suppose $h(z)=1+ze^z$ for $z\in \mathbb{D}$ and $\Omega=h(\mathbb{D})$. Let us define $\xi:\mathbb{C}^3\times \mathbb{D}\rightarrow \mathbb{C}$ as $\xi(\eta_1,\eta_2,\eta_3;z)=1+\alpha_1 \eta_2+\alpha_2 \eta_3$. For $\xi\in \Psi[\Omega, \Delta_e]$, we need to have $\xi(r,s,t;z)\notin \Omega$. Through \cite[Lemma 3.3]{kumar-ganganiaCardioid-2021},
we observe that the smallest disk containing $\Omega$ is $\{w\in \mathbb{C}: |w-1|<e\}$. Thus
      $$  |\xi(r,s,t;z)-1|=|\alpha_1 s+\alpha_2 t|.$$
Analogous to Theorem \ref{8 thmsigmoid}, we have
\begin{align*}
       |\alpha_1 s+\alpha_2 t|&\geq b(\theta)(\alpha_1+\alpha_2 l(\theta))\\
       &\geq \frac{1}{e}(\alpha_1-\alpha_2)\\
       &\geq e.
\end{align*}
Clearly, $\xi(r,s,t;z)\notin\{w\in \mathbb{C}: |w-1|<e\}$ which suffices us to conclude that $\xi(r,s,t;z)\notin \Omega$. Therefore, $\xi\in \Psi[\Omega,\Delta_e]$ and thus $p(z)\prec e^z$ by using Lemma \ref{8 theorem1}.
\end{proof}

Now, we conclude this section by choosing $h(z)$ to be $1+\sinh^{-1}z$ and $e^z$ itself, followed by its subsequent corollary.
\begin{theorem}
Suppose $\alpha_1$, $\alpha_2>0$ and $2(\alpha_1-\alpha_2 )\geq \pi e$. Consider the analytic function $p$ in $\mathbb{D}$ such that $p(0)=1$ and
     $$   1+\alpha_1 zp'(z)+\alpha_2 z^2p''(z)\prec 1+\sinh^{-1} z$$
implies $p(z)\prec e^z$.
\end{theorem}

\begin{proof}
Suppose $h(z)=1+\sinh^{-1} z$ for $z\in \mathbb{D}$ and $h(\mathbb{D})=\Omega:=\{\delta\in\mathbb{C}:|\sinh (\delta-1)|<1\}$. We define $\xi:\mathbb{C}^3\times \mathbb{D}\rightarrow \mathbb{C}$ as $\xi(\eta_1,\eta_2,\eta_3;z)=1+\alpha_1 \eta_2+\alpha_2 \eta_3$. For $\xi\in \Psi[\Omega, \Delta_e]$, we must have $\xi(r,s,t;z)\notin \Omega$. Through \cite[Remark 2.7]{kush}, we note that the disk $\{w\in\mathbb{C}: |w-1|<\pi/2\}$ is the smallest disk containing $\Omega$. So,
     $$   |\xi(r,s,t;z)-1|=|\alpha_1 s+\alpha_2 t|.$$
Analogous to Theorem \ref{8 thmsigmoid}, we have
\begin{align*}
       |\alpha_1 s+\alpha_2 t|&\geq b(\theta)(\alpha_1+\alpha_2 l(\theta))\\
       &\geq \frac{1}{e}(\alpha_1-\alpha_2)\\
       &\geq \frac{\pi}{2}.
\end{align*}
Clearly, $\xi(r,s,t;z)$ does not belong to the disk $\{w\in \mathbb{C}: |w-1|<\pi/2\}$ which suffices us to conclude that $\xi(r,s,t;z)\notin \Omega$. Therefore, $\xi\in \Psi[\Omega,\Delta_e]$ and thus $p(z)\prec e^z$ by using Lemma \ref{8 theorem1}.
\end{proof}

\begin{theorem}
Suppose $\alpha_1$, $\alpha_2>0$ and $\alpha_1-\alpha_2\geq e(e-1)$. Consider the analytic function $p$ in $\mathbb{D}$ such that $p(0)=1$ and
     $$   1+\alpha_1 zp'(z)+\alpha_2 z^2p''(z)\prec e^z$$
implies $p(z)\prec e^z$.
\end{theorem}

\begin{proof}
Suppose $h(z)=e^z$ for $z\in \mathbb{D}$ and $\Omega=\Delta_e$. Define $\xi:\mathbb{C}^3\times \mathbb{D}\rightarrow \mathbb{C}$ as $\xi(\eta_1,\eta_2,\eta_3;z)=1+\alpha_1 \eta_2+\alpha_2 \eta_3$. For $\xi\in \Psi[\Delta_e, \Delta_e]$, we must have $\xi(r,s,t;z)\notin \Delta_e$. So,
       $$ |\xi(r,s,t;z)-1|=|\alpha_1 s+\alpha_2 t|.$$
Analogous to Theorem \ref{8 thmsigmoid}, we have
\begin{align}\label{8 82}
       |\alpha_1 s+\alpha_2 t|&\geq b(\theta)(\alpha_1+\alpha_2 l(\theta))\nonumber\\
       &\geq \frac{1}{e}(\alpha_1-\alpha_2)\nonumber\\
       &\geq e-1.
\end{align}
Further, we have
   $$  |\log (\xi(r,s,t;z)|=|\log (1+\alpha_1 s+\alpha_2 t)|.$$
Through Lemma \ref{prilemma4} and \eqref{8 82}, we have
  $$ |\log (1+\alpha_1 s+\alpha_2 t)|\geq 1, $$
which implies that $\xi(r,s,t;z)\notin \Delta_e$. Therefore, $\xi\in \Psi[\Delta_e,\Delta_e]$ and thus $p(z)\prec e^z$ through Lemma \ref{8 theorem1}.
\end{proof}

Theorems 4.5-4.10 yield the following conclusion if $p(z)=zf'(z)/f(z)$ is considered as in \eqref{8 corollary21}.
\begin{corollary}
Suppose $\alpha_1$, $\alpha_2>0$ and $f\in\mathcal{A}$. Then, $f\in \mathcal{S}^{*}_{e}$ if any of the conditions hold:
\begin{enumerate}[$(i)$]
\item $Y_{f}(z)\prec 2/(1+e^{-z})$ and $\alpha_1-\alpha_2\geq e r_0$, where $r_0\approx 0.546302$ is the positive root of the equation $r^2+2 \cot (1)r-1=0$.
\item $Y_{f}(z)\prec z+\sqrt{1+z^2}$ and $\alpha_1-\alpha_2\geq \sqrt{2}e$.
\item $Y_{f}(z)\prec 1+\sin z$ and $\alpha_1-\alpha_2\geq e\sinh 1$.
\item $Y_{f}(z)\prec 1+ze^z$ and $\alpha_1-\alpha_2\geq e^2$.
\item $Y_{f}(z)\prec 1+\sinh^{-1}z$ and $2(\alpha_1-\alpha_2 )\geq \pi e$.
\item $Y_{f}(z)\prec e^z$ and $\alpha_1-\alpha_2\geq e(e-1)$.
\end{enumerate}
\end{corollary}


Lastly, we conclude this article by determining the sufficient conditions on $\alpha_1$, $\alpha_2$ and $\alpha_3\in \mathbb{R}^{+}$ involving the constants $m$ and $k$ (defined in Lemma \ref{lemmaformk}) to satisfy the third order differential subordination implication, given by
  $$  1+\alpha_1 z p'(z)+\alpha_2 z^2p''(z)+\alpha_3 z^3p'''(z)\prec h(z)\implies p(z)\prec e^z.$$
We obtain the following result upon considering $h(z)=(1+Cz)/(1+Dz)$. 

\begin{theorem}\label{8 thirdorderfirst}
Suppose $\alpha_1$, $\alpha_2$, $\alpha_3>0$ and $(\alpha_1-\alpha_2-m^2\alpha_3-3m(k-1)\alpha_3)(1-D^2)\geq e(1+|D|)(C-D)$ where $-1<D<C\leq 1$. Consider the analytic function $p$ in $\mathbb{D}$ such that $p(0)=1$ and
  $$  1+\alpha_1 z p'(z)+\alpha_2 z^2p''(z)+\alpha_3 z^3p'''(z)\prec \frac{1+Cz}{1+Dz}$$
implies $p(z)\prec e^z$.
\end{theorem}

\begin{proof}
Suppose $h(z)=(1+Cz)/(1+Dz)$ for $z\in \mathbb{D}$ and
  $$ h(\mathbb{D})= \Omega:=\bigg\{\delta\in \mathbb{C}:\bigg|\delta-\frac{1-CD}{1-D^2}\bigg|<\frac{C-D}{1-D^2}\bigg\}.$$
Define $\xi:\mathbb{C}^4\times \mathbb{D}\rightarrow\mathbb{C}$ as $\xi(\eta_1,\eta_2,\eta_3,\eta_4;z)=1+\alpha_1 \eta_2+\alpha_2 \eta_3+\alpha_3 \eta_4$. For $\xi\in\Psi[\Omega,\Delta_e]$, it is needed that $\xi(r,s,t,u;z)\notin \Omega$ for $z\in \mathbb{D}$. Consider
\begin{align*}
   \bigg| \xi(r,s,t,u;z)-\frac{1-CD}{1-D^2}\bigg|&=\bigg|1+\alpha_1 s+\alpha_2 t+\alpha_3 u-\frac{1-CD}{1-D^2}\bigg|\\
   &\geq |\alpha_1 s+\alpha_2 t+\alpha_3 u|-\frac{|D|(C-D)}{1-D^2}\\
   &\geq |\alpha_1 s|\RE\bigg(1+\frac{\alpha_2}{\alpha_1}\frac{t}{s}+\frac{\alpha_3}{\alpha_1}\frac{u}{s}\bigg)-\frac{|D|(C-D)}{1-D^2}\\
   &\geq m\alpha_1 b(\theta)\bigg(1+\frac{\alpha_2}{\alpha_1}(ml(\theta)+m-1)+\frac{\alpha_3}{\alpha_1}(m^2h(\theta)+3m(k-1)l(\theta))\bigg)\\
   &\quad-\frac{|D|(C-D)}{1-D^2}.
\end{align*}
As $m\geq 1$ and $\RE(1+l(\theta))>0$, we get
\begin{align*}
    \bigg| \xi(r,s,t,u;z)-\frac{1-CD}{1-D^2}\bigg|&\geq \alpha_1 b(\theta)\bigg(1+\frac{\alpha_2}{\alpha_1}l(\theta)+\frac{\alpha_3}{\alpha_1}(m^2h(\theta)+3m(k-1)l(\theta))\bigg)-\frac{|D|(C-D)}{1-D^2}\\
    &\geq \frac{1}{e}\bigg(\alpha_1-\alpha_2-\alpha_3(m^2+3m(k-1))\bigg) -\frac{|D|(C-D)}{1-D^2}\\
    &\geq \frac{C-D}{1-D^2}.
\end{align*}
From the above calculations, we conclude that $\xi\in \Psi[\Omega,\Delta_e]$ and hence $p(z)\prec e^z$ through Lemma \ref{8 firstlemmathirdorder}.
\end{proof}

\begin{corollary}
Suppose $\alpha_1$, $\alpha_2$, $\alpha_3>0$ and $f\in\mathcal{A}$. Let
\begin{align}
\chi_f(z)&:=1+\alpha_1 A_1+(\alpha_1+2\alpha_2)(A_2-A_1^2)+(\alpha_2+3\alpha_3)(2A_1^3-3A_1A_2+3A_3)\nonumber\\
 &\quad+\alpha_3(A_4-3A_2^2-(6A_1)^4-4A_1A_3+12A_1^2A_2),  \label{8 corollary31}
\end{align} 
where $A_1$, $A_2$, $A_3$ are same as defined in Corollary \ref{8 a1a2a3} and $A_4:=z^4f^{(iv)}(z)/f(z)$.
Then, $f(z)\in \mathcal{S}^{*}_{e}$ if
 $\chi_{f}(z)\prec (1+Cz)/(1+Dz)$ and $(\alpha_1-\alpha_2-m^2\alpha_3-3m\alpha_3(k-1))(1-D^2)\geq e(1+|D|)(C-D)$ whenever $-1< D<C\leq 1$.
    
\end{corollary}

\begin{theorem}
Suppose $\alpha_1$, $\alpha_2$ and $\alpha_3$ be positive real numbers and  $(\alpha_1-\alpha_2-m^2\alpha_3-3m(k-1)\alpha_3)(\alpha_1-\alpha_2-m^2\alpha_3-3m(k-1)\alpha_3-2e)\geq e^2$. Consider the analytic function $p$ in $\mathbb{D}$ such that $p(0)=1$ and
 $$   1+\alpha_1 z p'(z)+\alpha_2 z^2p''(z)+\alpha_3 z^3p'''(z)\prec \sqrt{1+z}$$
implies $p(z)\prec e^z$.
\end{theorem}

\begin{proof}
Suppose $h(z)=\sqrt{1+z}$ and $h(\mathbb{D})=\Omega:=\{\delta\in \mathbb{C}:|\delta^2-1|<1\}$. Consider $\xi:\mathbb{C}^4\times \mathbb{D}\rightarrow\mathbb{C}$ defined as $\xi(\eta_1,\eta_2,\eta_3,\eta_4;z)=1+\alpha_1 \eta_2+\alpha_2 \eta_3+\alpha_3 \eta_4$. We know that $\xi\in \Psi[\Omega,\Delta_e]$ if $\xi(r,s,t,u;z)\notin \Omega$ for $z\in \mathbb{D}$. Consider
\begin{align*}
    |(\xi(r,s,t,u;z))^2-1|&=|(1+\alpha_1 s+\alpha_2 t+\alpha_3 u)^2-1|\\
    &\geq |\alpha_1 s+\alpha_2 t+\alpha_3 u|(|\alpha_1 s+\alpha_2 t+\alpha_3 u|-2)\\
    &\geq |\alpha_1 s|\RE\bigg(1+\frac{\alpha_2}{\alpha_1}\frac{t}{s}+\frac{\alpha_3}{\alpha_1}\frac{u}{s}\bigg)\bigg(|\alpha_1 s|\RE\bigg(1+\frac{\alpha_2}{\alpha_1}\frac{t}{s}+\frac{\alpha_3}{\alpha_1}\frac{u}{s}\bigg)-2\bigg).
\end{align*}
Analogous to the proof of Theorem \ref{8 thirdorderfirst} and the fact that $m\geq 1$, we obtain
\begin{align*}
  |(\xi(r,s,t,u;z))^2-1|&\geq \alpha_1 b(\theta)\bigg(1+\frac{\alpha_2}{\alpha_1}l(\theta)+\frac{\alpha_3}{\alpha_1}(m^2h(\theta)+3m(k-1)l(\theta))\bigg)\times\\
  &\quad\bigg(\alpha_1 b(\theta)\bigg(1+\frac{\alpha_2}{\alpha_1}l(\theta)+\frac{\alpha_3}{\alpha_1}(m^2h(\theta)+3m(k-1)l(\theta))\bigg)-2\bigg)\\
  &\geq\frac{1}{e^2}\bigg(\alpha_1-\alpha_2-\alpha_3(m^2+3m(k-1))\bigg)\bigg(\alpha_1-\alpha_2-\alpha_3(m^2+3m(k-1))-2e\bigg)\\
  &\geq 1.
\end{align*}
This implies that $\xi\in \Psi[\Omega,\Delta_e]$. Thus, $p(z)\prec e^z$ as a consequence of Lemma \ref{8 firstlemmathirdorder}.
\end{proof}

\begin{corollary}
By considering $p(z)=zf'(z)/f(z)$ in above theorem and through \eqref{8 corollary31}, we deduce that $f(z)\in \mathcal{S}^{*}_e$ if
    $\chi_{f}(z)\prec \sqrt{1+z}$ and $(\alpha_1-\alpha_2-m^2\alpha_3-3m\alpha_3(k-1))(\alpha_1-\alpha_2-m^2\alpha_3-3m(k-1)\alpha_3-2e)\geq e^2$.
\end{corollary}

The next two results deal with the case when $h(z)=2/(1+e^{-z})$ and $z+\sqrt{1+z^2}$.

\begin{theorem}\label{8 sigmoid3}
Suppose $\alpha_1$, $\alpha_2$, $\alpha_3>0$ and $\alpha_1-\alpha_2-m^2\alpha_3-3m(k-1)\alpha_3\geq e r_0$, where $r_0\approx 0.546302$ is the positive root of the equation $r^2+2 \cot (1)r-1=0$.  Consider the analytic function $p$ in $\mathbb{D}$ such that $p(0)=1$ and
$$1+\alpha_1 zp'(z)+\alpha_2 z^2p''(z)+\alpha_3 z^3p'''(z)\prec \frac{2}{1+e^{-z}} $$    
implies $p(z)\prec e^z$.
\end{theorem}

\begin{proof}
Suppose $h(z)=2/(1+e^{-z})$ and $h(\mathbb{D})=\Omega:=\{\delta\in \mathbb{C}:|\log (\delta/(2-\delta))|<1\}$. Define $\xi:\mathbb{C}^4\times \mathbb{D}\rightarrow \mathbb{C}$ as $\xi(\eta_1,\eta_2,\eta_3,\eta_4;z)=1+\alpha_1 \eta_2+\alpha_2 \eta_3+\alpha_3 \eta_4$. For $\xi\in \Psi[\Omega, \Delta_e]$, it is required that $\xi(r,s,t,u;z)\notin \Omega$. Consider
\begin{align*}
    |\alpha_1 s+\alpha_2 t+\alpha_3 u|&=\alpha_1 |s|\bigg|1+\frac{\alpha_2}{\alpha_1}\frac{t}{s}+\frac{\alpha_3}{\alpha_1}\frac{u}{s}\bigg|\\
    &\geq \alpha_1 |s|\RE\bigg(1+\frac{\alpha_2}{\alpha_1}\frac{t}{s}+\frac{\alpha_3}{\alpha_1}\frac{u}{s}\bigg)\\
    &\geq m \alpha_1 b(\theta)\bigg(1+\frac{\alpha_2}{\alpha_1}(ml(\theta)+m-1)+\frac{\alpha_3}{\alpha_1}(m^2h(\theta)+3m(k-1)l(\theta))\bigg).
\end{align*}
As $m\geq 1$ and $\RE (1+l(\theta))>0$, we have
\begin{align}\label{8 73}
    |\alpha_1 s+\alpha_2 t+\alpha_3 u|&\geq b(\theta)\bigg(\alpha_1+\alpha_2 l(\theta)+\alpha_3 (m^2h(\theta)+3m(k-1)l(\theta))\bigg)\nonumber\\
    &\geq \frac{1}{e}(\alpha_1-\alpha_2-m^2\alpha_3-3m(k-1)\alpha_3)\nonumber\\
    &\geq r_0.
\end{align}
Next, we consider
  $$  \bigg|\log \bigg(\frac{\xi(r,s,t,u;z)}{2-\xi(r,s,t,u;z)}\bigg)\bigg|=\bigg|\log\bigg(\frac{1+\alpha_1 s+\alpha_2 t+\alpha_3 u}{1-(\alpha_1 s+\alpha_2 t+\alpha_3 u)}\bigg)\bigg|.$$
Through Lemma \ref{prilemma42} and \eqref{8 73}, we have
  $$  \bigg|\log \bigg(\frac{1+\alpha_1 s+\alpha_2 t+\alpha_3 u}{1-(\alpha_1 s+\alpha_2 t+\alpha_3 u)}\bigg)\bigg|\geq 1.$$
Therefore, $\xi\in \Psi[\Omega,\Delta_e]$ and $p(z)\prec e^z$ as a consequence of Lemma \ref{8 firstlemmathirdorder}.
\end{proof}

\begin{theorem}
Suppose $\alpha_1$, $\alpha_2$, $\alpha_3>0$ and $\alpha_1-\alpha_2-m^2\alpha_3-3m(k-1)\alpha_3\geq \sqrt{2}e$. Consider the analytic function $p$ in $\mathbb{D}$ such that $p(0)=1$ and
     $$   1+\alpha_1 zp'(z)+\alpha_2 z^2p''(z)+\alpha_3 z^3p'''(z)\prec z+\sqrt{1+z^2}$$
implies $p(z)\prec e^z$.
\end{theorem}

\begin{proof}
Suppose $h(z)=z+\sqrt{1+z^2}$ for $z\in \mathbb{D}$ and thus $h(\mathbb{D})=\Omega:=\{\delta\in \mathbb{C}:|\delta^2-1|<2|\delta|\}$. Define $\xi:\mathbb{C}^4\times \mathbb{D}\rightarrow \mathbb{C}$ as $\xi(\eta_1,\eta_2,\eta_3,\eta_4;z)=1+\alpha_1 \eta_2+\alpha_2 \eta_3+\alpha_3 \eta_4$. For $\xi\in \Psi[\Omega, \Delta_e]$, we must have $\xi(r,s,t,u;z)\notin \Omega$. From the graph of $z+\sqrt{1+z^2}$ (see Fig. \ref{crescent}), we note that $\Omega$ is constructed by two circles $C_1$ and $C_2$, given by
        $$C_1:|z-1|=\sqrt{2}\quad\text{and}\quad C_2:|z+1|=\sqrt{2}.$$
It is obvious that $\Omega$ contains the disk enclosed by $C_1$ and excludes the portion of the disk enclosed by $C_2\cap C_1$. We have
      $$  |\xi(r,s,t,u;z)-1|=|\alpha_1 s+\alpha_2 t+\alpha_3 u|.$$
Similar to the proof of Theorem \ref{8 sigmoid3}, we have
\begin{align*}
       |\alpha_1 s+\alpha_2 t+\alpha_3 u|&\geq b(\theta)\bigg(\alpha_1+\alpha_2 l(\theta)+\alpha_3(m^2h(\theta)+3m(k-1)l(\theta))\bigg)\\
       &\geq \frac{1}{e}(\alpha_1-\alpha_2-m^2\alpha_3-3m(k-1)\alpha_3)\\
       &\geq \sqrt{2}.
\end{align*}
Thus, we can say that $\xi(r,s,t,u;z)$ does not belong to the circle $C_1$ which suffices us to deduce that $\xi(r,s,t,u;z)\notin \Omega$. Therefore, $\xi\in \Psi[\Omega,\Delta_e]$ and thus $p(z)\prec e^z$ by using Lemma \ref{8 firstlemmathirdorder}.
\end{proof}

If we take $h(z)$ to be $1+\sin z$ and $1+ze^z$, the results are as follow:
\begin{theorem}
Suppose $\alpha_1$, $\alpha_2$, $\alpha_3>0$ and $\alpha_1-\alpha_2-m^2\alpha_3-3m(k-1)\alpha_3\geq e\sinh 1$. Consider the analytic function $p$ in $\mathbb{D}$ such that $p(0)=1$ and
     $$   1+\alpha_1 zp'(z)+\alpha_2 z^2p''(z)+\alpha_3 z^3p'''(z)\prec 1+\sin z$$
implies $p(z)\prec e^z$.
\end{theorem}

\begin{proof}
Suppose $h(z)=1+\sin z$ and $h(\mathbb{D})=\Omega:=\{\delta\in \mathbb{C}:|\arcsin (\delta-1)|<1\}$. Let us define $\xi:\mathbb{C}^4\times \mathbb{D}\rightarrow \mathbb{C}$ as $\xi(\eta_1,\eta_2,\eta_3,\eta_4;z)=1+\alpha_1 \eta_2+\alpha_2 \eta_3+\alpha_3 \eta_4$. For $\xi\in \Psi[\Omega, \Delta_e]$, it is required that $\xi(r,s,t,u;z)\notin \Omega$. Through \cite[Lemma 3.3]{chosine}, we note that the smallest disk containing $\Omega$ is $\{w: |w-1|<\sinh 1\}$. So,
      $$  |\xi(r,s,t,u;z)-1|=|\alpha_1 s+\alpha_2 t+\alpha_3 u|.$$
Analogous to the proof of Theorem \ref{8 sigmoid3}, we have
\begin{align*}
       |\alpha_1 s+\alpha_2 t+\alpha_3 u|&\geq b(\theta)\bigg(\alpha_1+\alpha_2 l(\theta)+\alpha_3(m^2h(\theta)+3m(k-1)l(\theta))\bigg)\\
       &\geq \frac{1}{e}(\alpha_1-\alpha_2 -m^2\alpha_3-3m(k-1)\alpha_3)\\
       &\geq \sinh 1.
\end{align*}
Clearly, $\xi(r,s,t,u;z)$ does not belong to the disk $\{w: |w-1|<\sinh 1\}$ indicates that $\xi(r,s,t,u;z)\notin \Omega$. Therefore, $\xi\in \Psi[\Omega,\Delta_e]$ and thus $p(z)\prec e^z$ by employing Lemma \ref{8 firstlemmathirdorder}.
\end{proof}

\begin{theorem}
Suppose $\alpha_1$, $\alpha_2$, $\alpha_3>0$ and $\alpha_1-\alpha_2-m^2\alpha_3-3m(k-1)\alpha_3\geq e^2$. Consider the analytic function $p$ in $\mathbb{D}$ such that $p(0)=1$ and
    $$    1+\alpha_1 zp'(z)+\alpha_2 z^2p''(z)+\alpha_3 z^3p'''(z)\prec 1+ze^z$$
implies $p(z)\prec e^z$.
\end{theorem}

\begin{proof}
Suppose $h(z)=1+ze^z$ and $\Omega=h(\mathbb{D})$. We define $\xi:\mathbb{C}^4\times \mathbb{D}\rightarrow \mathbb{C}$ as $\xi(\eta_1,\eta_2,\eta_3,\eta_4;z)=1+\alpha_1 \eta_2+\alpha_2 \eta_3+\alpha_3 \eta_4$. For $\xi\in \Psi[\Omega, \Delta_e]$, we need to have that $\xi(r,s,t,u;z)\notin \Omega$. Through \cite[Lemma 3.3]{kumar-ganganiaCardioid-2021}, we note that the smallest disk containing $\Omega$ is $\{w\in\mathbb{C}: |w-1|<e\}$. Also,
      $$  |\xi(r,s,t,u;z)-1|=|\alpha_1 s+\alpha_2 t+\alpha_3 u|.$$
Analogous to the proof of Theorem \ref{8 sigmoid3}, we have
\begin{align*}
       |\alpha_1 s+\alpha_2 t+\alpha_3 u|&\geq b(\theta)\bigg(\alpha_1+\alpha_2 l(\theta)+\alpha_3(m^2h(\theta)+3m(k-1)l(\theta))\bigg)\\
       &\geq \frac{1}{e}(\alpha_1-\alpha_2-m^2\alpha_3-3m(k-1)\alpha_3)\\
       &\geq e.
\end{align*}
Clearly, $\xi(r,s,t,u;z)$ does not belong to the disk $\{w\in \mathbb{C}: |w-1|<e\}$, is enough to conclude that $\xi(r,s,t,u;z)\notin \Omega$. Therefore, $\xi\in \Psi[\Omega,\Delta_e]$ and thus $p(z)\prec e^z$ by using Lemma \ref{8 firstlemmathirdorder}.
\end{proof}

We complete this section by considering $h(z)=1+\sinh^{-1}z$ and $e^z$ itself which are followed by a resulting corollary.
\begin{theorem}
Suppose $\alpha_1$, $\alpha_2$, $\alpha_3>0$ and $2(\alpha_1-\alpha_2 -m^2\alpha_3-3m(k-1)\alpha_3)\geq \pi e$. Consider the analytic function $p$ in $\mathbb{D}$ such that $p(0)=1$ and
     $$   1+\alpha_1 zp'(z)+\alpha_2 z^2p''(z)+\alpha_3 z^3p'''(z)\prec 1+\sinh^{-1} z$$
implies $p(z)\prec e^z$.
\end{theorem}

\begin{proof}
Suppose $h(z)=1+\sinh^{-1} z$ and $h(\mathbb{D})=\Omega:=\{\delta\in\mathbb{C}:|\sinh (\delta-1)|<1\}$. Define $\xi:\mathbb{C}^4\times \mathbb{D}\rightarrow \mathbb{C}$ as $\xi(\eta_1,\eta_2,\eta_3,\eta_4;z)=1+\alpha_1 \eta_2+\alpha_2 \eta_3+\alpha_3 \eta_4$. For $\xi\in \Psi[\Omega, \Delta_e]$, it is required that $\xi(r,s,t,u;z)\notin \Omega$. Through \cite[Remark 2.7]{kush}, we note that the smallest disk containing $\Omega$ is $\{w\in\mathbb{C}: |w-1|<\pi/2\}$. So,
      $$  |\xi(r,s,t,u;z)-1|=|\alpha_1 s+\alpha_2 t+\alpha_3 u|.$$
Similar to the proof of Theorem \ref{8 sigmoid3}, we have
\begin{align*}
       |\alpha_1 s+\alpha_2 t+\alpha_3 u|&\geq b(\theta)\bigg(\alpha_1+\alpha_2 l(\theta)+\alpha_3(m^2h(\theta)+3m(k-1)l(\theta))\bigg)\\
       &\geq \frac{1}{e}(\alpha_1-\alpha_2-m^2\alpha_3-3m(k-1)\alpha_3)\\
       &\geq \frac{\pi}{2}.
\end{align*}
Clearly, $\xi(r,s,t,u;z)\notin\{w\in \mathbb{C}: |w-1|<\pi/2\}$ which suffices to prove that $\xi(r,s,t,u;z)\notin \Omega$. Therefore, $\xi\in \Psi[\Omega,\Delta_e]$ and thus $p(z)\prec e^z$ by Lemma \ref{8 firstlemmathirdorder}.
\end{proof}

\begin{theorem}
Suppose $\alpha_1$, $\alpha_2$, $\alpha_3>0$ and $\alpha_1-\alpha_2-m^2\alpha_3-3m(k-1)\alpha_3\geq e(e-1)$. Consider the analytic function $p$ in $\mathbb{D}$ such that $p(0)=1$ and
       $$ 1+\alpha_1 zp'(z)+\alpha_2 z^2p''(z)+\alpha_3 z^3p'''(z)\prec e^z$$
implies $p(z)\prec e^z$.
\end{theorem}

\begin{proof}
Suppose $h(z)=e^z$ and $\Omega=\Delta_e$. We define $\xi:\mathbb{C}^4\times \mathbb{D}\rightarrow \mathbb{C}$ as $\xi(\eta_1,\eta_2,\eta_3,\eta_4;z)=1+\alpha_1 \eta_2+\alpha_2 \eta_3+\alpha_3 \eta_4$. For $\xi\in \Psi[\Delta_e, \Delta_e]$, it is necessary to have $\xi(r,s,t,u;z)\notin \Delta_e$. So,
      $$  |\xi(r,s,t,u;z)-1|=|\alpha_1 s+\alpha_2 t+\alpha_3 u|.$$
Similar to Theorem \ref{8 sigmoid3}, we get
\begin{align}\label{8 83}
       |\alpha_1 s+\alpha_2 t+\alpha_3 u|&\geq b(\theta)\bigg(\alpha_1+\alpha_2 l(\theta)+\alpha_3(m^2h(\theta)+3m(k-1)l(\theta))\bigg)\nonumber\\
       &\geq \frac{1}{e}(\alpha_1-\alpha_2-m^2\alpha_3-3m(k-1)\alpha_3)\nonumber\\
       &\geq e-1.
\end{align}
Next, we consider
    $$ |\log (\xi(r,s,t,u;z)|=|\log (1+\alpha_1 s+\alpha_2 t+\alpha_3 u)|.$$
Through Lemma \ref{prilemma4} and \eqref{8 83}, we have
  $$ |\log (1+\alpha_1 s+\alpha_2 t+\alpha_3 u)|\geq 1, $$
 which implies that $\xi(r,s,t,u;z)\notin \Delta_e$. Therefore, $\xi\in \Psi[\Delta_e,\Delta_e]$ and thus $p(z)\prec e^z$. 
\end{proof}

Theorems 5.5-5.10 yield the following conclusion if $p(z)=zf'(z)/f(z)$ is considered as in \eqref{8 corollary31}.
\begin{corollary}
Suppose $\alpha_1$, $\alpha_2$ and $\alpha_3$ be positive real numbers with $f\in\mathcal{A}$. Then, $f(z)\in \mathcal{S}^{*}_{e}$ if any of these conditions hold:
\begin{enumerate}[$(i)$]
\item $\chi_{f}(z)\prec 2/(1+e^{-z})$ and $\alpha_1-\alpha_2-m^2\alpha_3-3m(k-1)\alpha_3\geq e r_0$, where $r_0\approx 0.546302$ is the positive root of the equation $r^2+2 \cot (1)r-1=0$.
\item $\chi_{f}(z)\prec z+\sqrt{1+z^2}$ and $\alpha_1-\alpha_2-m^2\alpha_3-3m(k-1)\alpha_3\geq \sqrt{2}e$.
\item $\chi_{f}(z)\prec 1+\sin z$ and $\alpha_1-\alpha_2-m^2\alpha_3-3m(k-1)\alpha_3\geq e\sinh 1$.
\item $\chi_{f}(z)\prec 1+ze^z$ and $\alpha_1-\alpha_2-m^2\alpha_3-3m(k-1)\alpha_3\geq e^2$.
\item $\chi_{f}(z)\prec 1+\sinh^{-1}z$ and $2(\alpha_1-\alpha_2 -m^2\alpha_3-3m(k-1)\alpha_3)\geq \pi e$.
\item $\chi_{f}(z)\prec e^z$ and $\alpha_1-\alpha_2-m^2\alpha_3-3m(k-1)\alpha_3\geq e(e-1)$.
\end{enumerate}
\end{corollary}

				%
\subsection*{Acknowledgment}
Neha Verma is thankful to the Department of Applied Mathematics, Delhi Technological University, New Delhi-110042 for providing Research Fellowship.


\begin{thebibliography}{99}

\bibitem{antoninoandmiller} J. Antonino\ and\ S.~S. Miller, Third-order differential inequalities and subordinations in the complex plane, Complex Var. Elliptic Equ. {\bf 56} (2011), no.~5, 439--454.
\bibitem{kush} K. Arora\ and\ S.~S. Kumar, Starlike functions associated with a petal shaped domain, Bull. Korean Math. Soc. {\bf 59} (2022), no.~4, 993--1010.
\bibitem{chosine} N.~E. Cho, V. Kumar, S. S. Kumar\and\ V. Ravichandran, Radius problems for starlike functions associated with the sine function, Bull. Iranian Math. Soc. {\bf 45} (2019), no.~1, 213--232.
\bibitem{goel} P. Goel\ and\ S. S. Kumar, Certain class of starlike functions associated with modified sigmoid function, Bull. Malays. Math. Sci. Soc. {\bf 43} (2020), no.~1, 957--991
\bibitem{goodman vol1} A. W. Goodman, Univalent functions. Vol. I, Mariner Publishing Co., Inc., Tampa, FL, 1983.
\bibitem{1janowski} W. Janowski, Extremal problems for a family of functions with positive real part and for some related families, Ann. Polon. Math. {\bf 23} (1970/71), 159--177.
\bibitem{ckms} P.~G. Krishnan, V. Ravichandran\ and\ P. Saikrishnan, Functions subordinate to the exponential function, Commun. Korean Math. Soc. {\bf 38} (2023), no.~1, 163--178.
\bibitem{goelhigher} S.~S. Kumar\ and\ P. Goel, Starlike functions and higher order differential subordinations, Rev. R. Acad. Cienc. Exactas F\'{\i}s. Nat. Ser. A Mat. RACSAM {\bf 114} (2020), no.~4, Paper No. 192, 23 pp.
\bibitem{kumar-ganganiaCardioid-2021} S. S. Kumar\ and\ G. Kamaljeet, A cardioid domain and starlike functions, Anal. Math. Phys. {\bf 11} (2021), no.~2, Paper No. 54, 34 pp.
\bibitem{sushil} S. Kumar\ and\ V. Ravichandran, Subordinations for functions with positive real part, Complex Anal. Oper. Theory {\bf 12} (2018), no.~5, 1179--1191.
\bibitem{ma-minda} W. C. Ma\ and\ D. Minda, A unified treatment of some special classes of univalent functions, Proc. Confer. Complex Anal. (Tianjin, 1992), 157--169.
\bibitem{madaan} V. Madaan, A. Kumar\ and\ V. Ravichandran, Starlikeness associated with lemniscate of Bernoulli, Filomat {\bf 33} (2019), no.~7, 1937-1955.
\bibitem{mendi} R. Mendiratta, S. Nagpal\ and\ V. Ravichandran, On a subclass of strongly starlike functions associated with exponential function, Bull. Malays. Math. Sci. Soc. {\bf 38} (2015), no.~1, 365--386.
\bibitem{miller} S. S. Miller\ and\ P. T. Mocanu, Differential subordinations, Monographs and Textbooks in Pure and Applied Mathematics, 225, Marcel Dekker, Inc., New York, 2000.
\bibitem{mushtaq} S. Mushtaq, M. Raza\ and\ J. Sok\'{o}{\l}, Differential subordination related with exponential functions, Quaest. Math. {\bf 45} (2022), no.~6, 889--899.
\bibitem{adibastarlikenessexponential} A. Naz, S. Nagpal\ and\ V. Ravichandran, Starlikeness associated with the exponential function, Turkish J. Math. {\bf 43} (2019), no.~3, 1353--1371.
\bibitem{raina} R. K. Raina\ and\ J. Sok\'{o}l, Some properties related to a certain class of starlike functions, C. R. Math. Acad. Sci. Paris {\bf 353} (2015), no. 11, 973--978.
\bibitem{stan} J. Sok\'{o}\l\ and\ J. Stankiewicz, Radius of convexity of some subclasses of strongly starlike functions, Zeszyty Nauk. Politech. Rzeszowskiej Mat. No. 19 (1996), 101--105.












					
					
					
					
					
\end{thebibliography}
\end{document}